\documentclass{article}

\usepackage{epsfig, color}
\usepackage{amsfonts}
\usepackage{amssymb}
\usepackage{amsmath}   
\usepackage{amsthm}
\usepackage{latexsym}
\usepackage{enumerate}
\usepackage[subnum]{cases}
\usepackage{hyperref}
\usepackage{appendix}

\newtheorem{thm}{Theorem}[section]
\newtheorem{lemma}[thm]{Lemma}
\newtheorem{cor}[thm]{Corollary}
\newtheorem{prop}[thm]{Proposition}

\theoremstyle{remark}
\newtheorem{rmk}[thm]{Remark}

\theoremstyle{definition}
\newtheorem{defin}[thm]{Definition}

\numberwithin{equation}{section}

\DeclareMathOperator{\dist}{d}
\DeclareMathOperator*{\diam}{diam}

\def\hessxy{\D\Dbar}

\renewcommand{\epsilon}{\varepsilon}

\newcommand{\R}{\mathbb{R}}

\newcommand{\bdry}[1]{\partial{#1}}
\newcommand{\Omegabar}{\bar{\Omega}}
\DeclareMathOperator{\cl}{cl}
\def\xbar{\bar{x}}
\newcommand{\omclose}{\Omega^{\cl}}
\newcommand{\Omegabarclose}{\Omegabar^{\cl}}

\newcommand{\pdiff}[2][x]{\ensuremath{\frac{\partial}{\partial#1_{#2}}}}

\newcommand{\delzero}{\delta_0}
\newcommand{\Tinv}{T^{-1}}
\newcommand{\dsub}[2][]{d^{#2}_{#1}}
\newcommand{\dhat}[2][]{\hat{\dsub[#2]{#1}}}
\newcommand{\alphasub}{f}
\newcommand{\alphahat}{\hat{\alphasub}}
\newcommand{\boldd}{{\mathbf d}}
\newcommand{\phid}[1][\boldd]{\phi_{#1}}
\newcommand{\vis}[2][\boldd]{V_{#1}(#2)}
\newcommand{\G}[2][\boldd]{G(#1, #2)}
\newcommand{\GUp}[1]{G^{#1}}
\newcommand{\phidel}{\Phi_\delta}
\newcommand{\deltad}[1]{\Delta d^{#1}}
\newcommand{\deltanorm}{\lvert\Delta \boldd\rvert}
\newcommand{\bolddeltad}{\Delta {\bf d}}

\newcommand{\W}{W}
\newcommand{\Wdel}{\W^{\bolddeltad}}

\newcommand{\V}{V}
\newcommand{\p}{p}

\newcommand{\meas}[1]{\mu\left(#1\right)}
\newcommand{\calH}{\mathcal{H}}
\newcommand{\calN}{\mathcal{N}}
\newcommand{\surface}[1]{\calH^{\dim-1}\left(#1\right)}
\newcommand{\inner}[2]{\ensuremath{\langle#1, #2\rangle}}
\newcommand{\proj}[2][\dim]{\pi^{#1}(#2)}
\newcommand{\Leb}[1]{\left\vert{#1}\right\vert_{\mathcal{L}^n}}
\newcommand{\coord}[2]{\left[#1\right]_{#2}}
\newcommand{\Omegacoord}[1]{\coord{\Omega}{#1}}

\newcommand{\cExp}[2]{exp^c_{#1}({#2})}

\newcommand{\MTWcoord}[4]{-(c_{#1#2, \bar{p}\bar{q}}-c_{#1#2, \bar{r}}c^{\bar{r}, s}c_{s, \bar{p}\bar{q}})c^{\bar{p}, #3}c^{\bar{q}, #4}}

\def\D{D}
\def\Dbar{\bar{D}}
\def\E{E}
\def\Ebar{\bar{E}}
\def\M{M}
\def\Mbar{\bar{M}}
\DeclareMathOperator{\dVol}{dVol}
\def\ctil{\tilde{c}}
\def\I{I}
\def\maxsource{\sup_{\Omega}{\I}}
\def\maxconst{C}
\def\minconst{\Lambda}
\def\parameterinterval{M}

\newcommand{\Wcoord}[2]{\coord{\W_{#1}}{\xbar_{#2}}}
\newcommand{\Wdelcoord}[3][\bolddeltad]{\coord{\W^{#1}_{#2}}{\xbar_{#3}}}
\newcommand{\Wbdrycoord}[2]{\coord{\partial{\W_{#1}}}{\xbar_{#2}}}
\newcommand{\ombdrycoord}[1]{\coord{\partial\Omega}{#1}}
\newcommand{\plane}[1]{P^{#1}}

\def\nstop{{n_\epsilon}}
\newcommand\dbar[2][]{\bar{d}_{#1}^{#2}}
\def\dim{d}
\def\paralellconst{\alpha}
\newcommand{\distM}[2]{\dist_\M{\left({#1}, {#2}\right)}}
\newcommand{\distMbar}[2]{\dist_{\Mbar}{\left({#1}, {#2}\right)}}
\def\constone{C''}
\def\pbar{\bar{p}}
\newcommand{\norm}[1]{\lvert {#1}\rvert}

\author{Jun Kitagawa}
\title{An iterative scheme for solving the optimal transportation problem}
\begin{document}
\maketitle

\begin{abstract}
We demonstrate an iterative scheme to approximate the optimal transportation problem with a discrete target measure under certain standard conditions on the cost function. Additionally, we give a finite upper bound on the number of iterations necessary for the scheme to terminate, in terms of the error tolerance and number of points in the support of the discrete target measure.
\end{abstract}

\section{Introduction}
In~\cite{CKO99}, Caffarelli, Kochengin, and Oliker developed a numerical algorithm to calculate approximations to discretizations of the far-field reflector antenna problem, and also gave a finite upper bound on the number of steps necessary. Later, Xu-Jia Wang showed that the far-field reflector problem can be interpreted as an optimal transportation problem (see~\cite{Wan04}, and also~\cite{Loe10}). A number of similar algorithms have also been considered, each for a problem related to some form of optimal transportation problem:~\cite{OP88} proposes an algorithm to solve the classical Monge-Amp{\`e}re equation in two dimensions (see however, Remark~\ref{rmk: A3w possible} below), while~\cite{KO98} considers the near-field reflector problem where the target domain is contained in a flat domain. Such results lead to a natural question: can a similar numerical algorithm can be developed for more general optimal transportation problems (and perhaps more importantly, with a similar upper bound on the number of iterations)?

In this paper, we answer this question in the affirmative, given certain conditions on the optimal transportation cost function. Namely, we show that an iterative scheme similar to the one given by Caffarelli, Kochengin, and Oliker can be applied to optimal transportation problems associated to a cost function satisfying the conditions~\eqref{A0},~\eqref{twist},~\eqref{nondeg}, and~\eqref{A3s} (one can verify that the optimal transportation formulation of the far-field reflector problem satisfies the required conditions, see Section~\ref{section: prelim} for relevant definitions). Additionally, we give a finite upper bound on the number of iterations necessary for this scheme to halt, which is of the same form as what is proven in~\cite{CKO99}. We are careful to note, however, that we do not provide a bound on the actual computational complexity of the scheme, which may vary depending on the geometric details of the particular cost function (towards this direction, Kochengin and Oliker further explore the actual implementation of the algorithm for the far-field reflector problem in~\cite{KO03}).

The aforementioned conditions are natural in the theory of optimal transportation,~\eqref{A3s} was first introduced by Ma, Trudinger, and Wang in~\cite{MTW05}, where they show it is a key condition leading to a proof of regularity properties. Later, Loeper proved in~\cite{Loe09} that a weaker form of~\eqref{A3s} is actually necessary for regularity, and also showed  a number of geometric implications of~\eqref{A3s} and its weaker form (see also~\cite{KM10}). It is precisely these geometric implications that we exploit in this paper, namely that certain sublevel sets possess a generalized notion of strong convexity, and use this to show that calculations similar to Caffarelli, Kochengin, and Oliker can be applied to this more general case. 

The organization of this paper is as follows. In Section~\ref{section: prelim} we introduce the optimal transportation problem, and a number of basic conditions and results classical to the literature. In Section~\ref{section: main results} we state the main results that are claimed in this paper, while in Section~\ref{section: algorithm description} we give a detailed description of the aforemetioned iterative scheme. Section~\ref{section: derivative of G} is devoted to a series of calculations, which are invoked in Section~\ref{section: upper bound on steps} in order to provide the claimed upper bound~\eqref{eqn: error bound} on the number of iterations necessary.

\section{Preliminaries of Optimal Transportation}\label{section: prelim}
In this section, we recall some basic facts and definitions regarding the optimal transportation problem, along with the key conditions introduced in~\cite{MTW05}. For a (much) more comprehensive study of the optimal transportation problem, the interested reader is referred to~\cite{Vil09}.

Given two domains $\Omega$ and $\Omegabar$, and two probability measures $\mu$ and $\nu$ defined on them, along with a real valued cost function $c$ defined on $\omclose \times \Omegabarclose$, we wish to find a measurable mapping $T: \Omega \to \Omegabar$ satisfying $T_\# \mu = \nu$ (defined by $T_\# \meas{\Ebar}= \meas{\Tinv(\Ebar)}$ for all measurable $\Ebar\subseteq \Omegabar$) such that 
\begin{equation*}
\int_\Omega c(x, T(x))d\mu = \min_{S_\#\mu=\nu}\int_\Omega c(x, S(x))d\mu.
\end{equation*}
We will say that such a $T$ is a \emph{solution to the optimal transportation problem}, or a \emph{Monge solution}, for the cost function $c$ transporting the measure $\mu$ to the measure $\nu$.

Under mild conditions on $c$ and the measure $\mu$, it is known that a unique solution to this problem exists. For instance, it is sufficient if $\Omega$ and $\Omegabar$ are bounded subsets of Riemannian manifolds $\M$ and $\Mbar$ respectively, $c$ satisfies conditions~\eqref{A0},~\eqref{twist}, and~\eqref{nondeg} below, and $\mu$ is absolutely continuous with respect to the volume measure $\dVol_{\M}$ defined by the Riemannian metric on $\M$. Additionally, in this case $T$ can be determined from a scalar valued, Lipschitz continuous potential function $\phi$ by the formula
\begin{equation*}
T(x)=\cExp{x}{D\phi(x)}
\end{equation*}
where the differential $D\phi(x)$ exists in the $\dVol_{\M}$ almost-everywhere sense, and $\cExp{x}{p}$ is defined below (see~\cite[Chapter 10]{Vil09}).

Now let $\Omega$ and $\Omegabar$ be open, bounded domains with piecewise smooth boundaries in $\dim$--dimensional Riemannian manifolds $\M$ and $\Mbar$ respectively. We introduce a number of key conditions on the cost function $c$. Below, $\D c$ and $\Dbar c$ are the differential of $c$ in the $x$ and $\xbar$ variable respectively, which unambiguously defines the mapping 
\begin{align*}
-\hessxy{c(x, \xbar)}: T_{x}\Omega \to T_{-\Dbar c(x, \xbar)}\left(T^*_{\xbar}\Omegabar\right)\cong T^*_{\xbar}\Omegabar.
\end{align*}
\\
\noindent\underline{\bf{Smoothness of cost function:}}

Assume that
\begin{equation}\label{A0}
c \in C^{4}(\omclose \times \Omegabarclose).\tag{Reg}\\
\end{equation}

\noindent\underline{\bf{Twist condition:}} 

$c$ satisfies~\eqref{twist} if for each $x_0 \in\omclose$ and $\xbar_0 \in \Omegabarclose$, the mappings $\xbar \mapsto - \D c(x_0, \xbar)$ and $x \mapsto - \Dbar c(x, \xbar_0)$ are injective.

For any $\pbar \in - \D c(x_0, \Omegabar)$ and $x_0 \in \Omega$, (resp. $p \in - \Dbar c(\Omega, \xbar_0)$ and $\xbar_0 \in \Omegabar$) we write $\cExp{x_0}{\pbar}$ (resp. $\cExp{\xbar_0}{p}$) for the unique element of $\Omegabar$ (resp. $\Omega$) such that 
\begin{equation}
\begin{aligned}
- \D c(x_0, \cExp{x_0}{\pbar})&=\pbar,\\
- \Dbar c(\cExp{\xbar_0}{p}, \xbar_0)&=p.
\end{aligned}
\tag{Twist}
\label{twist}
\end{equation}
\begin{rmk}\label{rmk: sets in cotangent coordinates}
For fixed $x\in\Omega$ and $\xbar\in\Omegabar$, we will denote the representations of sets $\E\subset\Omega$ and $\Ebar\subset\Omegabar$ in the cotangent spaces above $x$ and $\xbar$ by
\begin{align*}
\coord{\E}{\xbar}:&=-\Dbar c(\E, \xbar),\\
\coord{\Ebar}{x}:&=-D c(x,\Ebar).
\end{align*}
\end{rmk}

\noindent\underline{\bf{Nondegeneracy condition:}}

$c$ satisfies~\eqref{nondeg} if the following linear mapping is invertible for every $x\in\Omega^{\cl}$ and $\xbar\in\Omegabar^{\cl}$:
\begin{equation}
-\hessxy{c(x, \xbar)}: T_{x}\Omega\to T^*_{\xbar}\Omegabar. \tag{Nondeg}
\label{nondeg}\\
\end{equation}

\noindent\underline{\bf{Strong MTW condition:}} 

A cost $c$ satisfies~\eqref{A3s} if there exists some $\delzero>0$ such that 
\begin{equation}\label{A3s}\tag{$\textrm{MTW}_+$}
\MTWcoord{i}{j}{k}{l}(x, \xbar)V^iV^j\eta_k\eta_l\geq \delzero\lvert V\rvert^2\lvert\eta\rvert^2,
\end{equation}
for any $x\in\Omega^{\cl}$, $\xbar\in\Omegabar^{\cl}$, and $V\in T_{x}\Omega$ and $\eta\in T^*_{x}\Omega$ such that $\eta(V)=0$. Here all derivatives are with respect to a fixed coordinate system, regular indices denote derivatives of $c$ with respect to the first variable, while indices with a bar above denote derivatives with respect to the second derivative, and a pair of raised indices denotes the matrix inverse.

Additionally, we need the following concepts of $c$-convexity of a domain, and $c$-convex functions.
\begin{defin}\label{def:c-convex domains}
We say that a set $\E\subseteq\Omega$ is \emph{(strongly) $c$-convex with respect to $\xbar_0\in\Omegabar$} if the set $\coord{\E}{\xbar_0}$ is a (strongly) convex subset of $T^*_{\xbar_0}\Omegabar$. 

Similarly, we say $\Ebar\subseteq\Omegabar$ is \emph{(strongly) $c$-convex with respect to $x_0\in\Omega$} if the set $\coord{\Ebar}{x_0}$ is a (strongly) convex subset of $T^*_{x_0}\Omega$.
%
\end{defin}
\begin{rmk}
For some fixed $\xbar\in\Omegabar$, given any two points $p_1:=- \Dbar c(x_1, \xbar)$ and $p_2:=- \Dbar c(x_2, \xbar)$ with $x_1$, $x_2\in\Omega$, we define the \emph{$c$-segment with respect to $\xbar$ between $x_1$ and $x_2$} as the image of the straight line segment between $p_1$ and $p_2$ under the map $\cExp{\xbar}{\cdot}$. It is clear that $\E$ is $c$-convex with respect to $\xbar\in\Omegabar$ if and only if every $c$-segment with respect to $\xbar$ between any two $x_1$ and $x_2\in\E$ remains inside $\E$. An symmetric statement and definition holds with the roles of $\Omega$ and $\Omegabar$ reversed.
\end{rmk}

\begin{defin}\label{def: c-convex functions}
We say that a function $\phi$ is \emph{$c$-convex}, if for every $x_0\in \Omega$ there exist $\xbar_0\in\Omegabar$ and $\lambda_0 \in \R$ such that 
\begin{align*}
\phi(x_0)&=-c(x_0, \xbar_0)+\lambda_0,\\
\phi(x) &\geq -c(x, \xbar_0)+\lambda_0
\end{align*}
 for all $x\neq x_0$. 
%

We also call such a function $-c(\cdot, \xbar_0)+\lambda_0$ that satisfies the above equality and inequality, a \emph{$c$-support function to $\phi$ at $x_0$}.
\end{defin}

\section{Main Result}\label{section: main results}
We state in this section, the main result of the paper.

Assume that $\mu:=\I\dVol_{\M}$ for some positive, real valued $\I\in C^\infty(\Omega^{\cl})$ which satisfies
\begin{equation*}
\int_\Omega \I(x)\dVol_{\M}(x)=1,
\end{equation*}
while the domains $\Omega$ and $\Omegabar$ are bounded subsets of Riemannian manifolds $\M$ and $\Mbar$ respectively. Additionally, assume the cost function $c$ satisfies conditions~\eqref{A0},~\eqref{twist},~\eqref{nondeg}, and~\eqref{A3s}, fix an integer $K\geq 2$ and $\{\xbar_i\}_{i=1}^{K}\subseteq\Omegabar$, a finite collection of $K$ distinct points, and $\{\alphasub_i\}_{i=1}^K$, a collection of $K$ real numbers satisfying $\sum_{i=1}^K \alphasub_i=1$ and $0<\alphasub_i<1$. At this point, we make the additional assumption that $\Omega$ is $c$-convex with respect to $\{\xbar_i\}_{i=1}^K$ and $\Omegabar$ is $c$-convex with respect to $\Omega$ (however, we do not make any assumptions on the support of $\I$).

The main result we present here is as follows: for any $\epsilon >0$, there is an iterative scheme to find numbers $\alphahat_i>0$ and $\dhat{i}\in \R$ such that 
\begin{equation}\label{eqn: error bound}
\lvert \alphahat_i-\alphasub_i\rvert < \epsilon,\qquad\forall 1\leq i \leq K,
\end{equation}
and $T(x) = \cExp{x}{D\phi(x)}$ is the Monge solution for the cost function $c$ transporting the measure $\mu$ to the measure $\displaystyle\sum_{i=1}^K \alphahat_i \delta_{\xbar_i}$, where
\begin{equation}\label{eqn: algorithm construction}
\phi(x):=\max_{1\leq i \leq K}{[-c(x, \xbar_i)+\dhat{i}]}.
\end{equation}
Additionally, we show that the number of iterations $\nstop$ necessary to find $\{\alphahat_i\}_{i=1}^K$ and $\{\dhat{i}\}_{i=1}^K$ is bounded above:
\begin{equation}\label{eqn: step bound}
\nstop\leq K\left[ \frac{K\maxconst \parameterinterval\maxsource}{\delta \minconst}\left[\max_{1\leq i\leq K}{\surface{\bdry{\Omegacoord{\xbar_i}}}}\right]+1\right],
\end{equation}
where
\begin{align}
\maxconst:&=\max_{1\leq i\neq k\leq K}{\sup_{x\in\Omega}{\frac{\lvert\det{\left(-\hessxy{c(x, \xbar_i)}\right)}\rvert}{\lvert \left(-\hessxy{c(x, \xbar_i)}\right)^{-1}\left(-\D  c(x, \xbar_i)+ \D c(x, \xbar_k)\right)\rvert_{T_{\xbar_i}\Omegabar}}}},\notag\\
\parameterinterval:&=\max_{1<k\leq K}{\sup_{x\in\Omega}{e^{-c(x, \xbar_k)+c(x, \xbar_1)}}}+1,\notag\\
\minconst:&= \min_{1<k\leq K}{\inf_{x\in\Omega}{e^{-c(x, \xbar_k)+c(x, \xbar_1)}}},\notag\\
\delta:&=\min{\left\{\frac{\epsilon}{K-1},\ \frac{\alphasub_1}{K}\right\}}\label{eqn: delta restriction}
\end{align}
and $\surface{\cdot}$ denotes $(\dim-1)$--dimensional Hausdorff measure.
\begin{rmk}\label{rmk: compare to CKO}
The upper bound here has the same dependency on the number of points $K$ in the target and the error $\epsilon$ as the one given in~\cite{CKO99}.
Note that the bound we give here is on the number of \emph{iterations} (the number of $c$-convex functions $\phi_n$ necessary, see the description in Section~\ref{section: algorithm description} below). However, since there are exactly $K$ intermediate steps per iteration, counting the total number of steps we would multiply by an additional factor of $K$. By the choice of $\delta$, we note that the total number of steps with accuracy $\epsilon$ has an upper bound that is essentially of the order of $K^4/\epsilon$. 
\end{rmk}
\begin{rmk}\label{rmk: A3w possible}
We would like to point out here that it is possible to obtain \emph{some} finite upper bound on $\nstop$ under the degenerate version of~\eqref{A3s} where the constant $\delta_0=0$ (this is the case, for example, with the cost function $c(x, \xbar)=\lvert x-y\rvert^2$ on $\R^n$, which is known to correspond to the classical Monge-Amp{\`e}re equation). However, the difference will be that the higher order terms in the expression of $\GUp{i}(\boldd+\bolddeltad)-\GUp{i}(\boldd)$ may only have order $O(\deltanorm)$ instead of $o(\deltanorm)$ as $\deltanorm\to 0$ (see the proofs of Propositions~\ref{prop: difference of G^k} and~\ref{prop: bound on derivative of G}). As a result, the corresponding terms in the derivative of $\GUp{i}$ may not vanish, but can be bounded above, and the final upper bound may contain $\sum_{i=1}^K{K\choose i}$ more terms.
\end{rmk}

\section{Description of the iterative scheme}\label{section: algorithm description}
In describing the algorithm, we follow most of the notation used by Caffarelli, Kochengin, and Oliker in~\cite{CKO99} to better highlight the parallels between the two schemes.
\begin{defin}\label{def: visibility set}
Let $\phi$ be a $c$-convex function on $\Omega$. Then for any measurable set $\Ebar \subseteq \Omegabar$, we define the \emph{visibility set of $\omega$ associated to $\phi$} by
\begin{align*}
\vis[\phi]{\Ebar}:&=\left\{x\in\Omega\mid\text{there exist }\lambda\in\R\text{ and }y\in\Ebar\text{ such that, }\right.\\
&\qquad\left.-c(\cdot, \xbar)+\lambda\text{ is a }c\text{-support function to }\phi\text{ at }x\right\}.
\end{align*}
We also define 
\begin{equation*}
\G[\phi]{\Ebar}:=\meas{\vis[\phi]{\Ebar}}.
\end{equation*}
\end{defin}
\begin{defin}
Write $\boldd :=(\dsub{1}, \dsub{2}, \ldots, \dsub{K})$ for $\dsub{i}>0$, $1\leq i \leq K$ (we will use the shorthand $\boldd>0$ to denote this from now on). Then we define the following function:
\begin{equation*}
\phid(x):=\max_{1\leq i \leq K}{[-c(x, \xbar_i)-\log{\dsub{i}}]}.
\end{equation*}
\end{defin}
We will generally use superscripts to denote coordinates for the remainder of the paper.

Since we are only concerned with the optimal transportation problem when $\nu$ is a finite sum of delta measures, we make some notational simplifications as follows. For any $1
\leq i\leq K$ we will write
\begin{align*}
\vis{\xbar_i}&:=\vis[\phid]{\left\{\xbar_i\right\}},\\
\GUp{i}(\phi)&:=\G[\phi]{\left\{\xbar_i\right\}},\\
\GUp{i}(\boldd)&:=\G[\phid]{\left\{\xbar_i\right\}},
\end{align*}
for any $c$-convex function $\phi$ and $\boldd>0$, when the collection of points $\{\xbar_i\}_{i=1}^K$ is fixed.

\begin{rmk}\label{rmk: monge solutions from potential}
Since $c$ satisfies conditions~\eqref{A0} and~\eqref{twist}, $\phid$ is differentiable $\dVol_{\M}$--a.e on $\Omega$, and we can define the map
\begin{align*}
T_{\boldd}(x):=\cExp{x}{D\phid(x)}
\end{align*}
 for $\dVol_{\M}$--a.e. $x\in\Omega$. Additionally, we can see that 
 \begin{align*}
T_{\boldd}(x)=\xbar_i 
\end{align*}
for all $x\in\vis{\xbar_i}$, and since $\phid$ is clearly $c$-convex, this implies that $T_{\boldd}$ is the unique Monge solution from the measure $\I\dVol_{\M}$ to $\sum_{i=1}^K \GUp{i}(\boldd)\delta_{\xbar_i}$ (\cite[Chapter 10]{Vil09}). Intuitively the quantities $\GUp{i}(\boldd)$ give the amount of mass in $\Omega$ that is transported by the map $T_{\boldd}$ to the point $\xbar_i$ for each $1\leq i\leq K$.
\end{rmk}

We will make also use of one of the consequences of a theorem proved by Loeper (\cite[Theorem 3.7]{Loe10}).
\begin{thm}[Loeper]\label{thm: Loeper}
If $c$ satisfies~\eqref{A0},~\eqref{twist},~\eqref{nondeg}, and~\eqref{A3s}, and $\Omega$ is $c$-convex with respect to $\Omegabar$, then for any $c$-convex function $\phi$ and $\xbar\in\Omegabar$, $\vis[\phi]{\xbar}$ is $c$-convex with respect to $\xbar$.
\end{thm}
%

Before describing the iterative scheme, we need to prove a monotonicity property of the functions $\GUp{i}$. To do so, we first give an alternate characterization of the sets $\vis{\xbar_i}$.

\begin{lemma}\label{lemma: characterization of vis sets}
Suppose $c$, $\mu$, $\Omega$, and $\Omegabar$ satisfy all of the conditions of Section~\ref{section: main results}, then for any $1\leq i\leq K$
\begin{equation*}
\vis{\xbar_i}=\V_{\boldd, i}
\end{equation*}
if $\V_{\boldd, i}\neq\emptyset$, where
\begin{equation*}
\V_{\boldd, i}:=\left\{x\in\Omega\mid -c(x, \xbar_i)-\log{\dsub{i}}\geq -c(x, \xbar_k)-\log{\dsub{k}},\ \forall1\leq k\leq K\right\}.
\end{equation*}

Otherwise, 
\begin{equation*}
\meas{\vis{\xbar_i}}=0.
\end{equation*}
\end{lemma}

\begin{proof}
Note that for any choice of $\boldd$, $\phid$ is a $c$-convex function on $\Omega$. It is clear from the definition of $\phid$ that, at any point in $\V_{\boldd, i}$ there exists a $c$-support function of the form $-c(\cdot, \xbar_i)+\lambda$ supporting to $\phid$ at that point, thus we obtain
\begin{equation*}
\V_{\boldd, i}\subseteq \vis{\xbar_i}.
\end{equation*}
%

If $\vis{\xbar_i}=\emptyset$, the claims of the lemma immediately follow, so let us suppose there exists some $x_0\in\vis{\xbar_i}$. Then there exists some $\lambda\in \R$ such that 
\begin{align*}
-c(x_0, \xbar_i)+\lambda &= \phid(x_0)\\
-c(x, \xbar_i)+\lambda&\leq \phid(x)
\end{align*}
for all $x\in\Omega$. In particular by the definition of $\phid$, 
\begin{align*}
-c(x_0, \xbar_i)+\lambda=\phid(x_0)&\geq -c(x_0, \xbar_i)-\log{\dsub{i}}\\
\implies\lambda&\geq-\log{\dsub{i}}.
\end{align*}

If there exists a point $x_1\in\V_{\boldd, i}$, we have from the definition of $\phid$ and $\V_{\boldd, i}$ that
\begin{align*}
-c(x_1, \xbar_i)+\lambda&\leq\phid(x_1)=-c(x_1, \xbar_i)-\log{\dsub{i}}\\
\implies \lambda&\leq -\log{\dsub{i}},
\end{align*}
hence $\lambda=-\log{\dsub{i}}$ and we obtain $x_0\in\V_{\boldd, i}$.

Now suppose $\V_{\boldd, i}=\emptyset$. Assume first that $\phid$ is differentiable at $x_0$. By the definition of $\phid$, this implies that for some $1\leq j_0\leq k$, $i\neq j_0$ we have
\begin{align*}
 \phid(x_0)&=-c(x_0, \xbar_{j_0})-\log{\dsub{j_0}},\\
  \phid(x)&\geq-c(x, \xbar_{j_0})-\log{\dsub{j_0}},\qquad\forall\;x\neq x_0
\end{align*}
and in particular
\begin{align*}
\D \phid(x_0)=-\D c(x_0, \xbar_{j_0}).
\end{align*}
At the same time, since $x_0\in \vis{\xbar_i}$, we must have 
\begin{align*}
\D \phid(x_0)=-\D c(x_0, \xbar_{i}).
\end{align*}
which contradicts~\eqref{twist}. However, since $\phid$ is differentiable $\dVol_{\M}$ almost everywhere (Remark~\ref{rmk: monge solutions from potential}), and $\mu$ is absolutely continuous with respect to $\dVol_{M}$, we obtain $\meas{\vis{\xbar_i}}=0$ as desired.
\end{proof}
%
%
Now we can easily see the following corollary.
\begin{cor}\label{cor: monotonicity of G^k}
Suppose $c$, $\mu$, $\Omega$, and $\Omegabar$ satisfy all of the conditions of Section~\ref{section: main results}. Also, fix an index $1\leq i\leq K$, another index $j\neq i$, and values $\dsub{k}>0$ for all $k\neq i$. Then, $\GUp{i}(\boldd)$ is decreasing in $\dsub{i}$, $\GUp{j}(\boldd)$ is increasing in $\dsub{i}$, and we have the following limits:
\begin{align}
\lim_{\dsub{i}\to 0}\GUp{i}(\boldd)&=1,\label{eqn: G^i grows}\\
\lim_{\dsub{i}\to 0}\GUp{j}(\boldd)&=0\label{eqn: G^j shrinks}.
\end{align}
\end{cor}
\begin{proof}
If all $\dsub{k}$ are fixed for $k\neq i$, we can easily see that the sets $\V_{\boldd, i}$ are increasing as $\dsub{i}$ decreases. Hence by Lemma~\ref{lemma: characterization of vis sets}
 above we immediately obtain the following monotonicity property for $\dsub[1]{i} < \dsub[2]{i}$:
\begin{equation*}
\meas{\vis[{\dsub{1}, \ldots, \dsub[1]{i}, \ldots, \dsub{K}}]{\xbar_i}}\geq\meas{\vis[{\dsub{1}, \ldots, \dsub[2]{i}, \ldots, \dsub{K}}]{\xbar_i}}.
\end{equation*}
Similarly, since $j\neq i$, we obtain when $\dsub[1]{i}<\dsub[2]{i}$:
\begin{equation*}
\meas{\vis[{\dsub{1}, \ldots, \dsub[1]{i}, \ldots, \dsub{K}}]{\xbar_j}} \leq\meas{\vis[{\dsub{1}, \ldots, \dsub[2]{i}, \ldots, \dsub{K}}]{\xbar_j}}.
\end{equation*}
Then recalling that $\GUp{i}(\boldd)=\meas{\vis{\xbar_i}}$, the claims of monotonicity are immediate.

To obtain the limiting values, note that from the boundedness of $c$ and the definition of $\phid$, if all $\dsub{k}$ are kept fixed for $k\neq i$, we will have $\V_{\boldd, i}=\Omega$, while $\V_{\boldd, j}=\emptyset$ for $\dsub{i}>0$ sufficiently small. Then, again by Lemma~\ref{lemma: characterization of vis sets} above, the limits~\eqref{eqn: G^i grows} and~\eqref{eqn: G^j shrinks} are immediate.
\end{proof}

With this monotonicity property in hand, we are ready to describe the iterative scheme in detail.

If $K=1$, we have that $-c(\cdot, \xbar_1)+\dhat{}$ gives rise to an optimal solution from $\mu$ to $\delta_{\xbar_1}$ for any choice of $\dhat{}\in\R$ (the associated mapping is simply $T(x)\equiv \xbar_1$). If $K=2$, we let
\begin{align*}
 \phi(x):=\max{\{-c(x, \xbar_1),\ -c(x, \xbar_2)+\dhat{}\}}
\end{align*}
then adjust $\dhat{}$ continuously until $\alphahat_1:=\GUp{1}(\boldd)$ and $\alphahat_2:=\GUp{2}(\boldd)$ satisfy the desired bounds. Thus we assume that $K\geq 3$.
 
The algorithm now consists of starting with a $c$-convex function $\phi_0:=\phid[\boldd^0]$ for an appropriate choice of $\boldd^0$, then decreasing each of the parameters $\dsub{k}$ in turn until the desired mapping is found.

First, (after fixing $\epsilon >0$) we define $\delta>0$ by
\begin{equation*}
\delta:=\min{\left\{\frac{\epsilon}{K-1},\ \frac{\alphasub_1}{K}\right\}}
\end{equation*}
(the second restriction will play a role in showing the upper bound on the number of steps necessary, see Section~\ref{section: upper bound on steps}). Then, define the set $\phidel$ by
\begin{equation*}\label{def: phidel}
\phidel :=\left\{\phid\mid \boldd>0,\ \GUp{i}(\boldd)\leq \alphasub_i+\delta\text{ for all }2\leq i\leq K \right\}.
\end{equation*}
Since $\alphasub_i+\delta> 0$ for each $1\leq i\leq K$, by the limiting value~\eqref{eqn: G^j shrinks} in Corollary~\ref{cor: monotonicity of G^k} we see that $\phi_{\boldd}\in\phidel$ whenever  $\dsub{1}>0$ is sufficiently small, in particular $\phidel\neq\emptyset$. Take any element of $\phidel$ and let it be $\phi_0$. We now construct a sequence of $c$-convex functions $\phi_n\in\phidel$ as follows.

Suppose we have $\phi_n\in\phidel$, we construct $\phi_{n+1}\in\phidel$ by first constructing a sequence of $K$ intermediate $c$-convex functions. Let $\phi_{n, 1}:=\phi_n$. Then, for any $1\leq i\leq K-1$, suppose
\begin{align*}
\phi_{n, i}=\phi_{(\dsub[{n, i}]{1}, \dsub[{n, i}]{2}, \ldots, \dsub[{n, i}]{K})}\in\phidel.
\end{align*}
If $\lvert \GUp{i+1}(\phi_{n, i})-\alphasub_{i+1}\rvert < \delta$ we simply set $\phi_{n, i+1}:=\phi_{n, i}$. Otherwise, since $\phi_{n, i}\in\phidel$ and $i+1\geq 2$, we must have
\begin{align*}
\GUp{i+1}(\phi_{n, i})\leq \alphasub_{i+1}-\delta.
\end{align*}
Now, $\GUp{i+1}(\boldd)$ is continuous in $\dsub{i+1}$ by Proposition~\ref{prop: bound on derivative of G} in Section~\ref{section: derivative of G}, and has the monotonicity property described in Corollary~\ref{cor: monotonicity of G^k}. Since $\alphasub_{i+1}<1$, and we have the limit~\eqref{eqn: G^i grows}, we can find a $0<\dbar[{n, i}]{i+1}<\dsub[{n, i}]{i+1}$ such that 
\begin{align*}
\GUp{i+1}({\dsub[{n, i}]{1}, \dsub[{n, i}]{2}, \ldots, \dbar[{n, i}]{i+1}, \ldots, \dsub[{n, i}]{K}})\in (\alphasub_{i+1}, \alphasub_{i+1}+\delta).
\end{align*}
Taking $\phi_{n, i+1}:=\phi_{(\dsub[{n, i}]{1}, \dsub[{n, i}]{2}, \ldots, \dbar[{n, i}]{i+1}, \ldots, \dsub[{n, i}]{K})}$, we see by Corollary~\ref{cor: monotonicity of G^k} again that $\phi_{n, i+1}\in\phidel$. We continue in this manner for each $1\leq i\leq K-1$ until we determine $\phi_{n+1}:=\phi_{n, K}$.

If it happens that on the $\nstop$ iteration we have 
\begin{align*}
\phi_{\nstop}=\phi_{\nstop, 1}=\ldots=\phi_{\nstop, K}=\phi_{\nstop+1},
\end{align*}
this would imply that 
 \begin{align*}
 \lvert \GUp{i}(\phi_\nstop)-\alphasub_{i}\rvert < \delta\leq\frac{\epsilon}{K-1}<\epsilon
 \end{align*}
  for all $2\leq i\leq K$ by the choice of $\delta$. At the same time, 
\begin{align*}
\lvert\GUp{1}(\phi_\nstop)-\alphasub_1\rvert&=\lvert 1-\sum_{i=2}^K\GUp{i}(\phi_\nstop)-1+\sum_{i=2}^K\alphasub_i\rvert\\
&< (K-1)\delta\leq \epsilon.
\end{align*}
Thus we can see that if we let 
\begin{align*}
 \alphahat_i&:=\GUp{i}(\phi_{\nstop}),\\
 \dhat{i}&:=-\log{\dsub[\nstop]{i}},
\end{align*}
the construction~\eqref{eqn: algorithm construction} gives us exactly the $c$-convex function $\phi_\nstop$, while the $\alphahat_i$ satisfy the desired bound~\eqref{eqn: error bound}, and Remark~\ref{rmk: monge solutions from potential} implies that $\phi_\nstop$ gives rise to the  Monge solution with the desired properties. Notice that the value of $\dsub{1}$ remains fixed throughout the algorithm.

Now, it is \emph{a priori} possible that this scheme may continue for an infinite number of iterations. The following two sections are devoted to showing that this is not the case, and showing the upper bound~\eqref{eqn: step bound} on the number of iterations.

\section{Derivative of the map $G$}\label{section: derivative of G}
We now show that $\GUp{i}$ is differentiable in the $i$th variable, and obtain an upper bound for this partial derivative. This bound will be crucial in showing the desired upper bound~\eqref{eqn: step bound}.

Throughout this section, we will fix $\boldd>0$ and one particular index $1\leq i \leq K$. Now define 
\begin{align*}
\bolddeltad &:=(\deltad{1}, \ldots, \deltad{K})\in\R^K,\\
\deltanorm&:= \max_{1\leq k \leq K}{\lvert\deltad{k}\rvert},
\end{align*}
with $\deltad{i}=0$. It will be implicitly assumed that $\dsub{k}+\deltad{k}>0$ for each $1\leq k\leq K$.

We will also write for any $1\leq j\leq K$,
\begin{align*}
\W_{i, j, \boldd}=\W_j&:=\left\{x\in\Omega\mid -c(x, \xbar_j)-\log{\dsub{j}}\leq -c(x, \xbar_i)-\log{\dsub{i}}\right\},\\
\Wdel_{i, j, \boldd}=\Wdel_j&:=\left\{x\in\Omega\mid  -c(x, \xbar_j)-\log{(\dsub{j}+\deltad{j})\leq-c(x, \xbar_i)-\log{\dsub{i}}}\right\}
\end{align*}
\begin{rmk}\label{rmk: \W_j are c-convex}
Note that $\W_j$ is the visibility set $\vis[\phi]{\xbar_i}$ for the $c$-convex function $\phi(x) := \sup{\left\{-c(x, \xbar_i)-\log{\dsub{i}},\ -c(x, \xbar_j)-\log{\dsub{j}}\right\}}$. Hence, by Theorem~\ref{thm: Loeper} we see that for any $1\leq j\leq K$, $\W_j$ is $c$-convex with respect to $\xbar_i$.

We define these sets here because when we consider various difference quotients of $\GUp{i}$, we will obtain intersections of sets of the form $\Wdel_j\setminus\W_j$ with different indices.
\end{rmk}

First, a technical lemma. The author believes this is a well known fact, but in the interest of completeness a proof is provided here.
\begin{lemma}\label{lemma: monotonicity of surface measure of convex sets}
If $A\subseteq B$ are both bounded, convex sets in $\R^\dim$, then 
\begin{equation*}
\surface{\bdry{A}}\leq\surface{\bdry{B}}
\end{equation*}
where $\surface{\cdot}$ is $(\dim-1)$--dimensional Hausdorff measure.
\end{lemma}
\begin{proof}
First, if the affine dimension of $A$ is strictly less than $\dim-1$ then $\surface{\bdry{A}}=0$ and the claim is immediate. If the affine dimension of $A$ is $\dim$, since $A$ is convex, for each $p\in \partial A\setminus N$ there is a unique unit vector $v(p)$ such that $\inner{p'-p}{v(p)}\leq 0$ for all $p'\in A$, where $N\subset \partial A$ satisfies $\surface{N}=0$. If the affine dimension of $A$ is $\dim-1$, we fix one of the unit vectors that is normal to the $\dim-1$ dimensional affine hull of $A$, and choose $v(p)$ to be that vector (in this situation, $N=\emptyset$). In both cases, since $A\subseteq B$ and $B$ is convex and compact, if we define $\lambda(p):=\sup{\{\lambda\geq 0\mid p+\lambda v(p)\in B\}}$, we see that $\lambda(p)$ is finite and $p+\lambda(p)v(p)\in\partial B$. Thus, 
\begin{equation*}
\Psi(p):=p+\lambda(p)v(p)
\end{equation*} 
is a well-defined map from $\partial A\setminus N$ to $\partial B$.

We now claim that $\Psi$ is injective on $\partial A\setminus N$, and also
\begin{equation}\label{Psi is an expansion}
\lvert \Psi(p_1)-\Psi(p_2)\rvert \geq \lvert p_1-p_2\rvert
\end{equation}
for any $p_1$, $p_2\in\partial A\setminus N$. Since $\surface{N}=0$, by the definition of Hausdorff measure we see that this would prove the lemma.

To show injectivity, suppose $p_1\neq p_2$ are both in $\partial A\setminus N$ but 
\begin{equation*}
\Psi(p_1)=\Psi(p_2)=:q_0.
\end{equation*}
Then, 
\begin{align*}
\lvert p_2-q_0\rvert ^2&=\lvert p_2-\Psi(p_1)\rvert ^2\\
&=\lvert p_2-(p_1+\lambda(p_1)v(p_1))\rvert^2\\
&=\lvert p_2-p_1\rvert^2-2\lambda(p_1)\inner{p_2-p_1}{v(p_1)}+\lambda(p_1)^2\\
&>\lambda(p_1)^2\\
&=\lvert p_1-\Psi(p_1)\rvert^2\\
&=\lvert p_1-q_0\rvert^2.
\end{align*}
However, by reversing the roles of $p_1$ and $p_2$ above we obtain the opposite strict inequality, hence $\Psi$ must be injective.

Now to prove the expansion property, first note that for any $p_1$, $p_2\in\partial A\setminus N$, 
\begin{align*}
\inner{p_2-p_1}{\Psi(p_1)-p_1}&=\lambda(p_1)\inner{p_2-p_1}{v(p_1)}\\
&\leq 0
\end{align*}
and similarly
\begin{equation*}
\inner{p_1-p_2}{\Psi(p_2)-p_2}\leq 0.
\end{equation*}
By adding these two equations together, we obtain
\begin{align*}
0&\geq \inner{p_2-p_1}{\Psi(p_1)-p_1}+\inner{p_2-p_1}{p_2-\Psi(p_2)}\\
&=\inner{p_2-p_1}{\Psi(p_1)-\Psi(p_2)}+\lvert p_2-p_1\rvert^2,
\end{align*}
and by applying the Cauchy-Schwarz inequality,
\begin{align*}
\lvert p_2-p_1\rvert^2&\leq \inner{p_2-p_1}{\Psi(p_2)-\Psi(p_1)}\\
&\leq \lvert p_2-p_1\rvert\lvert\Psi(p_2)-\Psi(p_1)\rvert
\end{align*}
which readily implies inequality~\eqref{Psi is an expansion}.
\end{proof}

We now prove a key lemma, utilizing the~\eqref{A3s} condition. Here, we are exploiting the strong convexity of sublevelsets of $c$-functions under~\eqref{A3s}, which was proven originally by Loeper. Below, $\distM{\cdot}{\cdot}$ and $\distMbar{\cdot}{\cdot}$ are the geodesic distances given by the Riemannian metrics on $\M$ and $\Mbar$ respectively. 
\begin{lemma}\label{lemma: Wj are only tangential at isolated points}
Suppose $\mu$, $c$, $\Omega$, and $\Omegabar$ satisfy all of the conditions in Section~\ref{section: main results}. If $x_0\in\omclose$ is a point where 
\begin{align*}
-c(x_0, \xbar_i)+c(x_0, \xbar_j)&=-c(x_0, \xbar_i)+c(x_0, \xbar_k)\\
-\D c(x_0, \xbar_i)+\D c(x_0, \xbar_j)&=\paralellconst\left(-\D c(x_0, \xbar_i)+\D c(x_0, \xbar_k)\right)
\end{align*}
for some $\paralellconst\neq 0$ and $j\neq k$, then there exists some $r>0$ depending on $i$, $j$, and $k$ such that at least one of either
\begin{equation*}
-c(x, \xbar_i)+c(x, \xbar_j)\neq -c(x_0, \xbar_i)+c(x_0, \xbar_j)
\end{equation*}
or
\begin{equation*}
-c(x, \xbar_i)+c(x, \xbar_k)\neq -c(x_0, \xbar_i)+c(x_0, \xbar_k)
\end{equation*}
is true for all $0<\distM{x}{x_0} < r$.
\end{lemma}

\begin{proof}
By condition~\eqref{twist}, we find that $\paralellconst\neq 1$. Thus, (writing $\p_i:=-\D c(x_0, \xbar_i)$, $\p_j:= -\D c(x_0, \xbar_j)$, and $\p_k:= -\D c(x_0, \xbar_k)$, which are all distinct again by~\eqref{twist}) we see that $\p_i$, $\p_j$, and $\p_k$ are all collinear. Without loss of generality, assume that $\p_j$ lies on the line segment between $\p_i$ and $\p_k$. This means that $\xbar_j$ lies on the $c$-segment with respect to $x_0$ between $\xbar_i$ and $\xbar_k$. Thus by~\cite[Proposition 5.1]{Loe10} and the compactness of $\Omega^{\cl}$,
\begin{align*}
&\max\left\{-c(x, \xbar_i)+c(x_0, \xbar_i), -c(x, \xbar_k)+c(x_0, \xbar_k)\right\}\\
&\qquad\geq -c(x, \xbar_j)+c(x_0, \xbar_j)+C' \distMbar{\xbar_i}{\xbar_k}^2\distM{x}{x_0}^2-\gamma\distM{x}{x_0}^3
\end{align*}
for some $C'>0$, $\gamma>0$ and all $\distM{x}{x_0}$ sufficiently small. Since $\xbar_i\neq \xbar_l$, if $\distM{x}{x_0} <r$ for $r>0$ sufficiently small we will have $C'\distMbar{\xbar_i}{\xbar_k}^2\distM{x}{x_0}^2-\gamma\distM{x}{x_0}^3>0$, and hence either $-c(x, \xbar_i)+c(x_0, \xbar_i)\neq -c(x, \xbar_j)+c(x_0, \xbar_j)$ or $-c(x, \xbar_k)+c(x_0, \xbar_k)\neq -c(x, \xbar_j)+c(x_0, \xbar_j)$ as desired.
\end{proof}

The following lemma is the most computationally intensive of the paper. In calculating the difference quotient of $\GUp{i}$, we will encounter intersections of sets of the form $\Wdel_j\setminus\W_j$ or $\W_j\setminus\Wdel_j$. What we show is that an intersection of these sets with two different indices has $\mu$ measure that decays like $o(\deltanorm)$ as $\deltanorm\to 0$. This will allow us to eliminate most of the terms in the final expression of $\pdiff[d]{i}{\GUp{i}}$. 

Now, under assumption~\eqref{A3s}, each of the sets $\Wdel_j\setminus\W_j$ or $\W_j\setminus\Wdel_j$ are essentially the differences of two nested, strongly convex sets which decrease to $(\dim-1)$--dimensional sets as $\deltanorm\to 0$. If we consider the intersection of two such ``generalized annuli'' and the limiting sets intersect transversally, the decay rate of $o(\deltanorm)$  is readily imagined. In the case when the limiting sets intersect tangentially, we can apply Lemma~\ref{lemma: Wj are only tangential at isolated points} above to obtain the desired decay. The reader who is satisfied with the preceeding explanation may wish to skip over the proof of the following lemma.
\begin{lemma}\label{lemma: order of intersections of differences}
If $c$, $\mu$, $\Omega$, and $\Omegabar$ satisfy all of the conditions in Section~\ref{section: main results}, then for any $1\leq j\neq k\leq K$,
\begin{align*}
 \meas{\left(\Wdel_j\setminus\W_j\right)\cap\left(\Wdel_k\setminus\W_k\right)}&=o(\deltanorm),\\
 \meas{\left(\W_j\setminus\Wdel_j\right)\cap\left(\W_k\setminus\Wdel_k\right)}&=o(\deltanorm).
\end{align*}
as $\deltanorm\to 0$.
Here, the rate of decay $o(\deltanorm)$ may depend on $\boldd\in\R^K$.
\end{lemma}

\begin{proof}
By conditions~\eqref{twist},~\eqref{nondeg}, the boundedness of $\Omega$, and since $\mu$ is absolutely continuous with respect to $\dVol_{\M}$ with bounded density, for any measurable set $E\subset\Omega$ we have
\begin{align*}
 \meas{E}\lesssim\Leb{\coord{E}{\xbar_i}},
\end{align*}
where $\Leb{\cdot}$ is the volume measure induced on $T^*_{\xbar_i}\Mbar$ by the Riemannian metric on $\Mbar$. Hence, it will be sufficient to prove 
\begin{align*}
 \Leb{\left(\Wdelcoord{j}{i}\setminus\Wcoord{j}{i}\right)\cap\left(\Wdelcoord{k}{i}\setminus\Wcoord{k}{i}\right)}&=o(\deltanorm)
 \end{align*}
 (the second equality in the statement of the lemma follows by a symmetric argument). Now since $\deltad{i}=0$, this implies $\Wdel_{i}=\W_{i}$, hence we can assume that $j$, $k\neq i$. We then define 
\begin{align*}
X_i(\p):&=\cExp{\xbar_i}{\p}\\
 \ctil_{l, i}(\p):&=c(X_i(\p), \xbar_l)-c(X_i(\p), \xbar_i)
\end{align*}
for any $1\leq l\leq K$, so that
\begin{align}
\Wcoord{j}{i}&=\{\p\in\Omegacoord{\xbar_i}\mid-\ctil_{j, i}(\p)\leq \log{\dsub{j}}-\log{\dsub{i}}\}\notag\\
\Wcoord{k}{i}&=\{\p\in\Omegacoord{\xbar_i}\mid-\ctil_{k, i}(\p)\leq \log{\dsub{k}}-\log{\dsub{i}}\}\notag\\
\Wdelcoord{j}{i}&=\{\p\in\Omegacoord{\xbar_i}\mid-\ctil_{j, i}(\p)\leq \log{(\dsub{j}+\deltad{j})}-\log{\dsub{i}}\}\notag\\
\Wdelcoord{k}{i}&=\{\p\in\Omegacoord{\xbar_i}\mid-\ctil_{k, i}(\p)\leq \log{(\dsub{k}+\deltad{k})}-\log{\dsub{i}}\}\label{eqn: wcood expressions}.
\end{align}
If either $\deltad{j}$ or $\deltad{k}$ are non-positive, we would have 
\begin{align*}
 \left(\Wdelcoord{j}{i}\setminus\Wcoord{j}{i}\right)\cap\left(\Wdelcoord{k}{i}\setminus\Wcoord{k}{i}\right)=\emptyset,
\end{align*}
hence we may assume both quantities are strictly positive.

We will now show that the desired rate of decay holds on a neighborhood of each point contained in $\Wbdrycoord{j}{i}\cap\Wbdrycoord{k}{i}$. More precisely, we show that for each $\p\in\Wbdrycoord{j}{i}\cap\Wbdrycoord{k}{i}$, there exists a neighborhood $\calN_{\p}$ depending on $\boldd$ and the cost function $c$ such that
\begin{equation}\label{eqn: local bound}
\Leb{\left(\Wdelcoord{j}{i}\setminus\Wcoord{j}{i}\right)\cap\left(\Wdelcoord{k}{i}\setminus\Wcoord{k}{i}\right)\cap \calN_{\p}}=o(\deltanorm)
\end{equation}
as $\deltanorm\to 0$. There are a number of different cases to work through, depending on how the two sets $\Wbdrycoord{j}{i}$ and $\Wbdrycoord{k}{i}$ intersect. It is understood that whenever we take a new ``small neighborhood'' of a point, it will be contained in any previous such neighborhoods. Also, in each of the following cases, we assume by a translation that 
\begin{align*}
0\in\Wbdrycoord{j}{i}\cap\Wbdrycoord{k}{i}.
\end{align*}

\noindent\underline{{\bf Case 1}: (Intersection points in the interior of $\Omegacoord{\xbar_i}$)}

Suppose that $0$ is in the interior of $\Omegacoord{\xbar_i}$. We have two subcases.

\underline{{\bf Case 1a}: (Tangential intersection)}

First assume that the normal vectors to $\Wbdrycoord{j}{i}$ and $\Wbdrycoord{k}{i}$ at $0$ are parallel. Since $0$ is in the interior of $\Omegacoord{\xbar_i}$, in some small neighborhood of $0$,
\begin{align*}
\Wbdrycoord{j}{i}&\subseteq\left\{\p\in\Omegacoord{\xbar_i}\mid-\ctil_{j, i}(\p)=\log{\dsub{j}}-\log{\dsub{i}}\right\},\\
\Wbdrycoord{k}{i}&\subseteq\left\{\p\in\Omegacoord{\xbar_i}\mid-\ctil_{k, i}(\p)=\log{\dsub{l}}-\log{\dsub{i}}\right\}
\end{align*}
and $-D\ctil_{j, i}$ and $-D\ctil_{k, i}$ (nonzero by~\eqref{twist}) are normal vectors to $\Wbdrycoord{j}{i}$ and $\Wbdrycoord{k}{i}$ respectively. Thus by Lemma~\ref{lemma: Wj are only tangential at isolated points}, we can find some radius $r>0$ such that 
\begin{align}
\Wbdrycoord{j}{i}\cap\Wbdrycoord{k}{i}\cap B_r(0)=\{0\}\label{eqn: no other points in nbhd}.
\end{align}
By translating and rotating coordinates (denoted by $(\p^1, \ldots, \p^\dim)$), we may assume that the $(\dim-1)$--dimensional hyperplane normal to $  -D\ctil_{j, i}(0)$ is given by $\p^\dim=0$, with $  -D\ctil_{j, i}(0)$ in the positive $\p^\dim$ direction. Since $c$ is differentiable, $\Wbdrycoord{j}{i}$ is given as the graph of a real valued $C^1$ function of $\p'=(\p^1, \ldots, \p^{\dim-1})$ on a small neighborhood of $0$, call this function $\rho(\p')$. Consider the transformation given by 
\begin{equation*}
\p\mapsto (\p', \p^\dim-\rho(\p')),
\end{equation*}
which maps $\Wbdrycoord{j}{i}$ onto the hyperplane $\plane{\dim}:=\left\{\p^\dim=0\right\}$ in this neighborhood.
It is easy to see that the differential of this map is the identity matrix at the origin, and so it has nondegenerate Jacobian determinant on a small neighborhood of the origin. Thus, we may assume that 
\begin{align}
\Wbdrycoord{j}{i}&\subset \plane{\dim},\notag\\
\Wcoord{j}{i}&\subset \{\p^\dim\leq 0\}\label{eqn: wcoord below}
\end{align}
in some small neighborhood of $0$ while~\eqref{eqn: no other points in nbhd} continues to hold.

We will now bound the height of $\left(\Wdelcoord{j}{i}\setminus\Wcoord{j}{i}\right)\cap\left(\Wdelcoord{k}{i}\setminus\Wcoord{k}{i}\right)$ in the $\p^\dim$ direction in a small neighborhood of the origin, in terms of $\deltanorm$. Recalling that 
\begin{align*}
\inner{  -D\ctil_{j, i}(\p)}{e_\dim}=\lvert  -D\ctil_{j, i}(\p)\rvert> 0
\end{align*}
for $\p\in\Wbdrycoord{j}{i}$ by~\eqref{twist}, we can find some sufficiently small neighborhood $B_{r_0}(0)$ of the origin on which
\begin{equation*}
\inner{  -D\ctil_{j, i}(\p)}{e_\dim}\geq \frac{1}{2}\inner{  -D\ctil_{j, i}(\p')}{e_\dim}.
\end{equation*}
 Now fix any $\p\in \left(\Wdelcoord{j}{i}\setminus\Wcoord{j}{i}\right)\cap B_{r_0}(0)$ (so in particular, $\p'\in\Wbdrycoord{j}{i}$), then since $\deltad{j}>0$ and by~\eqref{eqn: wcoord below} we see that
\begin{align*}
\inner{\p-\p'}{e_\dim}\geq 0.
\end{align*}
Then we calculate,
\begin{align*}
\frac{\deltad{j}}{\dsub{j}}+o(\deltanorm)&=\log{(\dsub{j}+\deltad{j})}-\log{\dsub{j}}\\
&\geq -\ctil_{j, i}(\p)+\ctil_{j, i}(\p')\\
&=\int_0^1  \inner{-D\ctil_{j, i}(t\p+(1-t)\p')}{(\p-\p')}dt\\
&=\lvert \p-\p'\rvert\int_0^1\inner{  -D\ctil_{j, i}(t\p+(1-t)\p')}{e_\dim}dt\\
&\geq \frac{1}{2}\inner{  -D\ctil_{j, i}(\p')}{e_\dim}\vert \p-\p'\rvert\\
&=\frac{1}{2}\lvert  -D\ctil_{j, i}(\p')\rvert\lvert \p-\p'\rvert\\
&\gtrsim\lvert \p-\p'\rvert
\end{align*}
by~\eqref{twist} and the compactness of $\Omega$. In particular, for some $\constone>0$ depending on $\boldd$ and $c$,
\begin{equation*}
\lvert \p-\p'\rvert \leq \constone\deltanorm
\end{equation*}
for $\deltanorm$ small enough, and any $\p\in \left(\Wdelcoord{j}{i}\setminus\Wcoord{j}{i}\right)\cap B_{r_0}(0)$.
Thus
\begin{equation}\label{the difference is in a slab}
\left(\Wdelcoord{j}{i}\setminus\Wcoord{j}{i}\right)\cap B_{r_0}(0)\subseteq \left\{\lvert \p^\dim\rvert< \constone\deltanorm\right\}\cap B_{r_0}(0).
\end{equation}
We now claim that for $\calN_0:= B_{r_0}(0)$,
\begin{equation}\label{eqn: decay of diameter of projection}
\diam{\left(\proj{\left(\Wdelcoord{j}{i}\setminus\Wcoord{j}{i}\right)\cap\left(\Wdelcoord{k}{i}\setminus\Wcoord{k}{i}\right)}\cap\calN_0\right)} =o(1)
\end{equation}
as $\deltanorm\to 0$, where $\proj{\cdot}$ is orthogonal projection $\plane{\dim}$. Suppose not, then there is a constant $K_0>0$, a sequence of $\bolddeltad_n$ with $\lvert\bolddeltad_n\rvert\to 0$ as $n\to\infty$, and a sequence of points 
\begin{align*}
\p_{1, n},\ \p_{2, n}\in\left(\Wdelcoord[\bolddeltad_n]{j}{i}\setminus\Wcoord{j}{i}\right)\cap\left(\Wdelcoord[\bolddeltad_n]{k}{i}\setminus\Wcoord{k}{i}\right)\cap \calN_0
\end{align*}
such that $\lvert \p'_{1, n}-\p'_{2, n}\rvert > K_0$ for all $n$. By the boundedness of $\calN_0$, we may pass to subsequences and assume that $\p_{1, n}\to \p_1$ and $\p_{2, n}\to \p_2$ for some points $\p_1$ and $\p_2$ as $n\to\infty$. By the above calculations, we see that 
\begin{align*}
\lvert \p_{1, n}-\p'_{1, n}\rvert &< \constone\lvert\bolddeltad_n\rvert\\
\lvert \p_{2, n}-\p'_{2, n}\rvert &< \constone\lvert\bolddeltad_n\rvert
\end{align*}
 for each $n$. Hence, by letting $n\to\infty$ we see that $\p_1=\p'_1$ and $\p_2=\p'_2$, and thus 
 \begin{align*}
 \lvert \p_1-\p_2\rvert>K_0.
 \end{align*}
 However, by the continuity of $c$, it is clear that both $\p_1$ and $\p_2$ must be contained in $\Wbdrycoord{j}{i}\cap\Wbdrycoord{k}{i}\cap \calN_0$ which contradicts~\eqref{eqn: no other points in nbhd}.

Finally, by~\eqref{the difference is in a slab} the set $\left(\Wdelcoord{j}{i}\setminus\Wcoord{j}{i}\right)\cap\left(\Wdelcoord{k}{i}\setminus\Wcoord{k}{i}\right)\cap\calN_0$ is contained in a cylinder with $(\dim-1)$--dimensional base 
\begin{align*}
 \proj{\left(\Wdelcoord{j}{i}\setminus\Wcoord{j}{i}\right)\cap\left(\Wdelcoord{k}{i}\setminus\Wcoord{k}{i}\right)}\cap\calN_0
\end{align*}
and height $\constone\deltanorm$, and by~\eqref{eqn: decay of diameter of projection} we see this set must have measure $o(\deltanorm)$ as $\deltanorm\to 0$, thus obtaining~\eqref{eqn: local bound}.

\underline{{\bf Case 1b}: (Nontangential intersection)}

Now suppose the normal vectors to $\Wbdrycoord{j}{i}$ and $\Wbdrycoord{k}{i}$ at $0$ are not parallel. Again, translating, rotating, and straightening out $\Wbdrycoord{j}{i}$ near $0$, we assume that $-D\ctil_{j, i}(0)$ is in the positive $\p^\dim$ direction, and both of the inclusions~\eqref{eqn: wcoord below} hold in a small neighborhood of the origin.

Since the intersection point is in the interior of $\Omegacoord{\xbar_i}$, by~\eqref{A0} and~\eqref{twist}, there is a unique $(\dim-1)$--dimensional tangent hyperplane $T_0\Wbdrycoord{k}{i}$, to $\Wbdrycoord{k}{i}$ at $0$ and it does not coincide with $\plane{\dim}$. Since they do not coincide, the intersection of $T_0\Wbdrycoord{k}{i}$ and $\plane{\dim}$ form an $(\dim-2)$--dimensional linear subspace. Pick a basis $\{v_1, \ldots, v_{\dim-2}\}$ for this subspace, then add two vectors $v$ and $w$ to the collection such that $\{v_1, \ldots, v_{\dim-2}, v\}$ is a basis for $\plane{\dim}$ and $\{v_1, \ldots, v_{\dim-2}, w\}$ is a basis for $T_0\Wbdrycoord{k}{i}$. Clearly, $\{v_1, \ldots, v_{\dim-2}, v, w\}$ is a linearly independent collection of $\dim$ vectors. If we define a linear transformation $L$ by
\begin{align*}
L(v_l)&=v_l,\ 1\leq l \leq \dim-2,\\
L(v)&=v,\\
L(w)&=(0, \ldots, 0, 1),
\end{align*}
this $L$ is invertible, maps $\plane{\dim}$ onto itself, and maps $T_0\Wbdrycoord{k}{i}$ onto an $(\dim-1)$--dimensional linear subspace orthogonal to $\plane{\dim}$. By an additional rotation, we may assume this subspace is $\plane{\dim-1}:=\left\{\p^{\dim-1}=0\right\}$. Since $\Wbdrycoord{k}{i}$ is locally the graph of a $C^1$ function on $\plane{\dim-1}$, we may straighten $\Wbdrycoord{k}{i}$ on a small neighborhood $\calN_0$ of $0$ so it coincides with $\plane{\dim-1}$ (as in the first part of the proof of Case 1a), and this transformation will map $\plane{\dim}$ to itself (in particular, we will still have $\Wbdrycoord{j}{i}\cap\calN_0\subset \plane{\dim}$). Thus, we may assume that in some small neighborhood of the origin,~\eqref{eqn: wcoord below} holds and $\Wbdrycoord{k}{i}=\plane{\dim-1}$, while $\Wcoord{k}{i}\subset \{\p^{\dim-1}\leq 0\}$. Additionally, by restricting to a smaller neighborhood if necessary (which we still write as $\calN_0$), we obtain by similar reasoning to Case 1a above that 
\begin{align*}
&\left(\Wdelcoord{j}{i}\setminus\Wcoord{j}{i}\right)\cap\left(\Wdelcoord{k}{i}\setminus\Wcoord{k}{i}\right)\cap\calN_0\\
&\qquad\subset \left\{\lvert \p^{\dim-1}\rvert<\constone\deltanorm,\ \lvert \p^{\dim}\rvert<\constone\deltanorm\right\}\cap\calN_0,
\end{align*}
which has measure of order $o(\deltanorm)$, hence we obtain~\eqref{eqn: local bound} again.\\
\\
\underline{{\bf Case 2}: (Intersection points on $\ombdrycoord{\xbar_i}$)}

Suppose now that $0\in\ombdrycoord{\xbar_i}$. Again we consider two subcases.

\underline{{\bf Case 2a}: (``Fake boundary'')} 

Suppose that $-\ctil_{j, i}(0)\neq \log{\dsub{j}}-\log{\dsub{i}}$ (thus $-\ctil_{j, i}(0)< \log{\dsub{j}}-\log{\dsub{i}}$ by the definition of $\W_j$). By the continuity of $-\ctil_{j, i}$, we can find a small neighborhood $\calN_0$ of the origin that depends on $c$ and $\boldd$ on which both $-\ctil_{j, i}<\log{(\dsub{j}+\deltad{j})}-\log{\dsub{i}}$ and $-\ctil_{j, i}<\log{\dsub{j}}-\log{\dsub{i}}$ whenever $\deltanorm$ is sufficiently small, i.e.
\begin{align*}
 \left(\Wdelcoord{j}{i}\setminus\Wcoord{j}{i}\right)\cap\calN_0=\emptyset,
\end{align*}
and we immediately see~\eqref{eqn: local bound}.

A symmetric argument holds if $-\ctil_{k, i}(0)\neq \log{\dsub{k}}-\log{\dsub{i}}$.

\underline{{\bf Case 2b}: (Tangential intersection)} 

Now suppose that $-\ctil_{j, i}(0)= \log{\dsub{j}}-\log{\dsub{i}}=\log{\dsub{k}}-\log{\dsub{i}}=-\ctil_{k, i}(0)$, and the normal vectors $-D\ctil_{j, i}(0)$ and $-D\ctil_{k, i}(0)$ are parallel. Then, note that since $\Omegacoord{\xbar_i}$ is convex and bounded by assumption, there exists a Lipschitz mapping that straightens the boundary of $\Wcoord{j}{i}\cap\Omegacoord{\xbar_i}$ in a small neighborhood of the origin, with a controlled Lipschitz norm depending only on $\Omega$, $c$, and $\boldd$. As such, we may apply the same proof as Case 1a above to obtain~\eqref{eqn: local bound} (note that Lemma~\ref{lemma: Wj are only tangential at isolated points} is still valid when the point $x_0\in\bdry{\Omega}$).

\underline{{\bf Case 2c}: (Nontangential intersection)} 

For the final case, suppose that $-\ctil_{j, i}(0)= \log{\dsub{j}}-\log{\dsub{i}}=\log{\dsub{k}}-\log{\dsub{i}}=-\ctil_{k, i}(0)$, and the normal vectors $-D\ctil_{j, i}(0)$ and $-D\ctil_{k, i}(0)$ are not parallel. Since $c$ is $C^1$ up to the boundary of $\Omega$, we may extend both $-\ctil_{j, i}$ and $-\ctil_{k, i}$ outside of $\Omegacoord{\xbar_i}^{\cl}$ to a small neighborhood of $0$ in a $C^1$ manner. Then we may apply the same proof as in Case 1b to obtain~\eqref{eqn: local bound} (note that we do not require the use of Lemma~\ref{lemma: Wj are only tangential at isolated points} in this case).\\
\\
Now we can show the desired global equality. For each point $\p\in\Wbdrycoord{j}{i}\cap \Wbdrycoord{k}{i}$, there exists a neighborhood $\calN_{\p}$ that corresponds to $\p$ via one of the above steps, and we can use the compactness of $\Wbdrycoord{j}{i}\cap \Wbdrycoord{k}{i}$ to extract a finite cover $\calN_{\p_1}, \ldots \calN_{\p_N}$. It is easy to see that each of the sets $\Wdelcoord{j}{i}\setminus\Wcoord{j}{i}$ and $\Wdelcoord{k}{i}\setminus\Wcoord{k}{i}$ are contained in some neighborhoods of $\Wbdrycoord{j}{i}$ and $\Wbdrycoord{k}{i}$ respectively, whose diameter decreases to zero with $\deltanorm$. Thus we obtain
\begin{align*}
&\Leb{\left(\Wdelcoord{j}{i}\setminus\Wcoord{j}{i}\right)\cap\left(\Wdelcoord{k}{i}\setminus\Wcoord{k}{i}\right)}\\
&\leq \sum_{l=1}^N \Leb{\left(\Wdelcoord{j}{i}\setminus\Wcoord{j}{i}\right)\cap\left(\Wdelcoord{k}{i}\setminus\Wcoord{k}{i}\right)\cap\calN_{\p_l}}\\
&=o(\deltanorm)
\end{align*}
as $\deltanorm\to 0$, as desired.
\end{proof}
With this lemma in hand, we are finally ready to show the differentiability of $\GUp{i}$, and establish an upper bound for the magnitude of the derivative. The calculations have been split into two separate propositions, due to length.
\begin{prop}\label{prop: difference of G^k}
Assume $c$, $\mu$, $\Omega$, and $\Omegabar$ satisfy all of the conditions in Section~\ref{section: main results}. Then,
\begin{equation*}
\GUp{i}(\boldd+\bolddeltad)-\GUp{i}(\boldd)=\sum_{k=1}^K\frac{\beta_{i,k, \bolddeltad}\deltad{k}}{\dsub{k}}+o(\deltanorm),
\end{equation*}
where
\begin{align*}
\beta_{i, k, \bolddeltad}:=
\int_{U_{k,\boldd, \bolddeltad}}\I(X_i(\p))\lvert  -D\ctil_{k, i}(\p)\rvert ^{-1}\lvert\det{DX_i(\p)}\rvert d\surface{\p}
\end{align*}
and (using some of the same notation as the proof of Lemma~\ref{lemma: order of intersections of differences})
\begin{align*}
X_i(\p):&=\cExp{\xbar_i}{\p},\\
\ctil_{k, i}(\p)&:=c(X_i(\p), \xbar_k)-c(X_i(\p), \xbar_i),\\
U_{k, \boldd, \bolddeltad}&:=
\begin{cases}
\left\{\p\in\Omegacoord{\xbar_i}\mid-\ctil_{k, i}(\p)=\log{\dsub{k}}-\log{\dsub{i}}\right\}\cap \coord{\bigcap_{l\neq k}\Wdel_l}{\xbar_i}, &\deltad{k}\geq0\\
\left\{\p\in\Omegacoord{\xbar_i}\mid-\ctil_{k, i}(\p)=\log{(\dsub{k}+\deltad{k})}-\log{\dsub{i}}\right\}\cap  \coord{\bigcap_{l\neq k}\W_l}{\xbar_i}, &\deltad{k}<0.
\end{cases}
\end{align*}
\end{prop}

\begin{proof}
We can calculate, in the notation of Lemma~\ref{lemma: characterization of vis sets},
\begin{align*}
\GUp{i}(\boldd+\bolddeltad)-\GUp{i}(\boldd)&=\meas{\bigcap_{l=1}^K\Wdel_l}-\meas{\bigcap_{l=1}^K\W_l}\\
&=\meas{\vis[\boldd+\bolddeltad]{\xbar_i}}-\meas{\vis[\boldd]{\xbar_i}}\\\\
&=\meas{\V_{\boldd+\bolddeltad, i}}-\meas{\V_{\boldd, i}}\\
&=\meas{\bigcap_{l=1}^K\Wdel_l\setminus\bigcap_{k=1}^K\W_k}-\meas{\bigcap_{k=1}^K\W_k\setminus\bigcap_{l=1}^K\Wdel_l}.
\end{align*}
Now note that
\begin{align*}
&\meas{\bigcap_{l=1}^K\Wdel_l\setminus\bigcap_{k=1}^K\W_k}\\
&=\meas{\bigcup_{k=1}^K\left(\bigcap_{l=1}^K\Wdel_l\setminus\W_k\right)}\\
&=\sum_{k=1}^K\meas{\bigcap_{l=1}^K\Wdel_l\setminus\W_k}-\sum_{1\leq k_1\neq k_2\leq K}\meas{\left(\bigcap_{l=1}^K\Wdel_l\setminus\W_{k_1}\right)\cap\left(\bigcap_{l=1}^K\Wdel_l\setminus\W_{k_2}\right)}+\ldots\\
&\qquad+(-1)^K\sum_{1\leq k_1\neq k_2\neq\ldots\neq k_K\leq K}\meas{\left(\bigcap_{l=1}^K\Wdel_l\setminus\W_{k_1}\right)\cap\left(\bigcap_{l=1}^K\Wdel_l\setminus\W_{k_2}\right)\cap\right.\\
&\qquad\left.\ldots\cap\left(\bigcap_{l=1}^K\Wdel_l\setminus\W_{k_K}\right)}.
\end{align*}
For any fixed index $k$, we have that $\bigcap_{l=1}^K\Wdel_l\setminus\W_k\subseteq \Wdel_k\setminus\W_k$, hence by Lemma~\ref{lemma: order of intersections of differences} above, all terms except the first one in the last line of calculations above has a rate of decay $o(\deltanorm)$ as $\deltanorm\to0$. 
Since we obtain a similar expression for 
\begin{align*}
 \meas{\bigcap_{k=1}^K\W_k\setminus\bigcap_{l=1}^K\Wdel_l},
 \end{align*}
we find
\begin{align*}
&\GUp{i}(\boldd+\bolddeltad)-\GUp{i}(\boldd)\\
&=\sum_{k=1}^K\left[\meas{\bigcap_{l=1}^K\Wdel_l\setminus\W_k}-\meas{\bigcap_{l=1}^K\W_l\setminus\Wdel_k}\right]+o(\deltanorm)\\
&=\sum_{k=1}^K\left[\meas{\bigcap_{l\neq k}^K\Wdel_l\cap\left(\Wdel_k\setminus\W_k\right)}-\meas{\bigcap_{l\neq k}^K\W_l\cap\left(\W_k\setminus\Wdel_k\right)}\right]+o(\deltanorm).
\end{align*}

Now, fix an index $1\leq k\leq K$. First assume that $\deltad{k}\geq 0$, hence $\W_k\setminus \Wdel_k=\emptyset$. Then, by making the change of variable $\p=-\D c(x, \xbar_i)$, we see that
\begin{align*}
&\meas{\bigcap_{l\neq k}^K\Wdel_l\cap\left(\Wdel_k\setminus\W_k\right)}-\meas{\bigcap_{l\neq k}^K\W_l\cap\left(\W_k\setminus\Wdel_k\right)}\\
&=\meas{\bigcap_{l\neq k}^K\Wdel_l\cap\left(\Wdel_k\setminus\W_k\right)}\\
&=\int_{\bigcap_{l\neq k}^K\Wdel_l\cap\left(\Wdel_k\setminus \W_k\right)}\I(x)\dVol_{\M}(x)\\
&=\int_{\coord{\bigcap_{l\neq k}^K\Wdel_l\cap\left(\Wdel_k\setminus \W_k\right)}{\xbar_i}}\I(X_i(\p))\lvert\det{DX_i(\p)}\rvert d\p\\
&=\left(\ast\right).
\end{align*}
 Recalling~\eqref{eqn: wcood expressions}, we see that
\begin{align*}
&\coord{\bigcap_{l\neq k}^K\Wdel_l\cap\left(\Wdel_k\setminus \W_k\right)}{\xbar_i}\\
&\qquad=\left\{\p\in \coord{\bigcap_{l\neq k}^K\Wdel_l}{\xbar_i}\mid \log{\dsub{k}}-\log{\dsub{i}}< -\ctil_{k, i}(\p)\leq \log{(\dsub{k}+\deltad{k})}-\log{\dsub{i}}\right\}.
\end{align*}
Then since $-\ctil_{k, i}(\p)$ is a Lipschitz function on $\Omegacoord{\xbar_i}$, by applying the coarea formula, we obtain
\begin{align}
&\left(\ast\right)\notag\\
&= \int_{\log{\dsub{k}}-\log{\dsub{i}}}^{\log{(\dsub{j}+\deltad{k})}-\log{\dsub{i}}}\left(\int_{\left\{-\ctil_{k, i}(\p)=t\right\}\cap \coord{\bigcap_{l\neq k}\Wdel_l}{\xbar_i}}\I(X_i(\p))\lvert  -D\ctil_{k, i}(\p)\rvert ^{-1}\lvert\det{DX_i(\p)}\rvert d\surface{\p}\right)dt\notag\\
&=\left(\int_{U_{k, \boldd, \bolddeltad}}\I(X_i(\p))\lvert  -D\ctil_{k, i}(\p)\rvert ^{-1}\lvert\det{DX_i(\p)}\rvert d\surface{\p}\right)(\log{(\dsub{k}+\deltad{k})}-\log{\dsub{k}})+o(\deltanorm)\notag\\
&= \frac{1}{\dsub{k}}\left(\int_{U_{k, \boldd, \bolddeltad}}\I(X_i(\p))\lvert  -D\ctil_{k, i}(\p)\rvert ^{-1}\lvert\det{DX_i(\p)}\rvert d\surface{\p}\right)\deltad{k}+o(\deltanorm)\notag\\
&=\frac{\beta_{i,k, \bolddeltad}\deltad{k}}{\dsub{k}}+o(\deltanorm)\label{eqn: kth equality}
\end{align}
as $\deltanorm\to 0$.

If $\deltad{k}<0$ we have $\Wdel_k\setminus\W_k=\emptyset$ instead, and following similar calculations we again obtain the expression~\eqref{eqn: kth equality}.
By summing over $1\leq k\leq K$, the proof is completed.
%
\end{proof}

We are now ready to provide the necessary upper bound on $\pdiff[d]{i}$.
\begin{prop}\label{prop: bound on derivative of G}
$\GUp{i}$ is differentiable in $\dsub{i}$ and 
\begin{equation*}
\lvert \pdiff[d]{i}\GUp{i}\rvert\leq \frac{K\maxconst \maxsource\left[\surface{\bdry{\Omegacoord{\xbar_i}}}\right]}{\dsub{i}}
\end{equation*}
for each $1\leq i\leq K$, where
\begin{align*}
\maxconst:=\max_{1\leq k\neq l\leq K}{\sup_{x\in\Omega}{\frac{\lvert\det{\left(-\hessxy{c(x, \xbar_l)}\right)}\rvert}{\lvert \left(-\hessxy{c(x, \xbar_l)}\right)^{-1}\left(-\D  c(x, \xbar_l)+ \D c(x, \xbar_k)\right)\rvert_{T_{\xbar_l}\Omegabar}}}}.
\end{align*}
\end{prop}

\begin{proof}
Fix $\boldd>0$. Now it is easy to see that for any $\lambda >0$ we have $\GUp{i}(\boldd)=\GUp{i}(\lambda\boldd)$. We have used here~\eqref{twist},~\eqref{nondeg}, and the fact that $\I$ is bounded above to move the $o(t)$ terms outside of the summation. Thus, applying Proposition~\ref{prop: difference of G^k} above, we calculate for any $t$ with $\norm{t}<\dsub{i}$,
\begin{align}
&\GUp{i}(\dsub{1}, \ldots, \dsub{i}+t, \ldots, \dsub{K})-\GUp{i}(\boldd)\notag\\
&=\GUp{i}(\frac{\dsub{1}\dsub{i}}{\dsub{i}+t}, \ldots, \dsub{i}, \ldots, \frac{\dsub{K}\dsub{i}}{\dsub{i}+t})-\GUp{i}(\boldd)\notag\\
&=\GUp{i}(\dsub{1}-\frac{\dsub{1}}{\dsub{i}}t+o(t), \ldots, \dsub{i}, \ldots, \dsub{K}-\frac{\dsub{K}}{\dsub{i}}t+o(t))-\GUp{i}(\boldd)\notag\\
&=\sum_{k=1}^K\frac{\beta_{i,k,\bolddeltad_t}}{{\dsub{k}}}\left(-\frac{\dsub{k}}{\dsub{i}}t\right)+o(t)\notag\\
&=-\frac{t}{\dsub{i}}\sum_{k=1}^K\beta_{i,k,\bolddeltad_t}+o(t)\label{eqn: difference quotient}
\end{align}
as $t\to 0$, and with $\bolddeltad_t:=(-\frac{\dsub{1}}{\dsub{i}}t, \ldots, 0, \ldots, -\frac{\dsub{K}}{\dsub{i}}t)$ with the $0$ in the $i$th component. Then, by applying the dominated convergence theorem we first see that $\pdiff[d]{i}{\GUp{i}}$ exists. 

In order to obtain the claimed inequality, first note that by the convexity of $\Omegacoord{\xbar_i}$ (by assumption) and Theorem~\ref{thm: Loeper}, the level sets of $-\ctil$ are the boundaries of convex subsets of $\Omegacoord{\xbar_i}^{\cl}$ . Thus by Lemma~\ref{lemma: monotonicity of surface measure of convex sets} we have
\begin{align*}
\surface{U_{k, \boldd, \bolddeltad_t}} \leq \surface{\bdry{\Omegacoord{\xbar_i}}}
\end{align*}
for any $1\leq k\leq K$ and $\norm{t}$ small. Additionally, by the definition of $X_i$ and $\maxconst$, and~\eqref{twist} and~\eqref{nondeg}, we can calculate for any $1\leq k\leq K$, 
\begin{align*}\label{eqn: maxconst inequality}
 \lvert  -D\ctil_{k, i}(q)\rvert ^{-1}\lvert\det{DX_i(q)}\rvert\leq \maxconst,
\end{align*}
hence by combining with~\eqref{eqn: difference quotient} we obtain
\begin{align*}
\lvert\pdiff[d]{i}{\GUp{i}}\rvert&=\lvert\lim_{t\to 0}\frac{\GUp{i}(\dsub{1}, \ldots, \dsub{i}+t, \ldots, \dsub{K})-\GUp{i}(\boldd)}{t}\rvert\\
&\leq\frac{1}{\dsub{i}}\sum_{k=1}^K\lim_{t\to0}{\lvert\beta_{i,k,\bolddeltad_t}\rvert}\\
&\leq \frac{K\maxconst \maxsource \left[\surface{\bdry{\Omegacoord{\xbar_i}}}\right]}{\dsub{i}}
\end{align*}
as claimed.
\end{proof}

\section{An upper bound on the number of steps}\label{section: upper bound on steps}
Recall that when we initialize the scheme, we require an initial choice of $\boldd^0$ such that $\GUp{i}(\boldd^0)\leq \alphasub_i+\delta$ for all $2\leq i \leq K$. This can be attained if we make the choice of 
\begin{equation}\label{eqn: starting upper bound}
\begin{cases}
\dsub[0]{1}=1&\\
\dsub[0]{i}=\parameterinterval, & i\neq 1
\end{cases}
\end{equation}
where
\begin{equation*}
\parameterinterval:=\max_{1<k\leq K}{\sup_{x\in\Omega}{e^{-c(x, \xbar_k)+c(x, \xbar_1)}}}+1.
\end{equation*}
In fact, with this choice we have $\GUp{i}(\boldd^0)=0$ for all $2\leq i \leq K$. Now notice that in each step, we either do not change the value of each $\dsub{i}$ or decrease it, so we may assume these values as upper bounds for our choices of $\dsub{i}$ throughout the algorithm. Now, recalling that the value of $\dsub{1}$ does not change (hence $\dsub{1}\equiv 1$ throughout), and that at every step our choice of $\phid$ remains in the set $\phidel$, we can use~\eqref{eqn: delta restriction} to calculate 
\begin{align*}	
0&\leq \alphasub_1-K\delta\\
&< \alphasub_1-(K-1)\delta\\
&=1-\sum_{i=2}^K(\alphasub_i+\delta)\\
&\leq 1-\sum_{i=2}^K\GUp{i}(\boldd)\\
&=\GUp{1}(\boldd).
\end{align*}

Now if we write
\begin{equation*}
\minconst:= \min_{1<k\leq K}{\inf_{x\in\Omega}{e^{-c(x, \xbar_k)+c(x, \xbar_1)}}}
\end{equation*}
we can see from definition that $\vis{\xbar_1}=\emptyset$ if $\dsub{i}< \minconst$ for any $i\neq 1$, and by Lemma~\ref{lemma: characterization of vis sets} we would have $\GUp{1}(\boldd)=0$. Hence there is a lower bound of 
\begin{align}
 \dsub{i}\geq \minconst\label{eqn: dsub lower bound}
\end{align}
for any $\dsub{i}$ with $i\neq 1$.

Now, say we are at the $(n, i-1)$--st step of the algorithm, and we must decrease $\dsub[{n, i-1}]{i}$ to $\dbar[{n, i-1}]{i}$. In this case we can combine Proposition~\ref{prop: bound on derivative of G} with~\eqref{eqn: dsub lower bound} to obtain, 
\begin{align*}
\delta&\leq \alphasub_{i}-\GUp{i}(\phi_{n, i-1})\\
&\leq \GUp{i}(\phi_{n, i})-\GUp{i}(\phi_{n, i-1})\\
&=\GUp{i}(\dsub[{n, i-1}]{1}, \ldots, \dbar[{n, i-1}]{i}, \ldots, \dsub[{n, i-1}]{K})-\GUp{i}(\dsub[{n, i-1}]{1}, \ldots, \dsub[{n, i-1}]{i}, \ldots, \dsub[{n, i-1}]{K})\\
&\leq \sup\left\vert\pdiff[d]{i}\GUp{i}\right\vert(\dsub[{n, i-1}]{i}-\dbar[{n, i-1}]{i})\\
&\leq \frac{K\maxconst \maxsource }{\minconst}\left[\surface{\bdry{\Omegacoord{\xbar_i}}}\right](\dsub[{n, i-1}]{i}-\dbar[{n, i-1}]{i})
\end{align*}
in other words, we obtain a strictly positive lower bound for the magnitude of the decrease in the parameter $\dsub{i}$.

However, since only strictly positive values of $\dsub{i}$ are admissible and the parameters $\dsub{i}$ can only decrease, we see that each $\dsub{i}$ can only be decreased a finite number of times. The worst case scenario is when only one parameter is updated per iteration of the scheme, hence the maximum number of iterations is given by $K$ times the maximum number of times each $\dsub{i}$ can be updated. Thus by the initial choice~\eqref{eqn: starting upper bound}, we see that an upper bound on the number of iterations $\nstop$ that this algorithm can take is given by the claimed bound~\eqref{eqn: step bound}:
\begin{equation*}
\nstop\leq K\left[\frac{K\maxconst \parameterinterval\maxsource}{\delta \minconst}\left[\max_{1\leq i\leq K}{\surface{\bdry{\Omegacoord{\xbar_i}}}}\right]+1\right].
\end{equation*}
\bibliographystyle{plain}
\bibliography{mybiblio}

\end{document}